\documentclass[11pt]{amsart}

\usepackage{amsmath,amsthm,amssymb,amsfonts,ifpdf}
\usepackage{xcolor}

\usepackage{xargs}  
\usepackage[colorinlistoftodos,prependcaption,textsize=tiny,obeyFinal]{todonotes}
\newcommandx{\ebltodo}[2][1=]{\todo[linecolor=red,backgroundcolor=red!25,bordercolor=red,#1]{#2}}
{% begin code
  \color{olive}%
}%
{}% end code 

\ifpdf
\usepackage[pdftex,pdfstartview=FitH,pdfborderstyle={/S/B/W 1},%
colorlinks=true, linkcolor=blue, urlcolor=blue, citecolor=blue,%
pagebackref=true]{hyperref}
\else
\usepackage[hypertex]{hyperref}
\fi
\usepackage{enumitem}
\setenumerate{leftmargin=*}
\setitemize{leftmargin=*}
\usepackage{amscd}
\usepackage[latin1]{inputenc}
\DeclareMathAlphabet{\mathpzc}{OT1}{pzc}{m}{it}
\usepackage{bbm}

\numberwithin{equation}{section}

\swapnumbers
\newtheorem{thm}[subsection]{Theorem}

\newtheorem*{cor*}{Corollary}
\newtheorem{lemma}[subsection]{Lemma}

\newtheorem{propos}[subsection]{Proposition}

\newtheorem*{thm*}{Theorem}
\newtheorem*{thma*}{Theorem A}
\newtheorem*{thmb*}{Theorem B}
\newtheorem*{thmc*}{Theorem C}

\swapnumbers
\theoremstyle{definition}

\newtheorem*{remark}{Remark}

\newcounter{consta}

\newcounter{constk}

\newcounter{constc}

\newcounter{constE}

\newcounter{constd}

\makeatletter
\newcommand*\bigcdot{\mathpalette\bigcdot@{.5}}
\newcommand*\bigcdot@[2]{\mathbin{\vcenter{\hbox{\scalebox{#2}{$\m@th#1\bullet$}}}}}
\makeatother

\def\XXint#1#2#3{{\setbox0=\hbox{$#1{#2#3}{\int}$ }
\vcenter{\hbox{$#2#3$ }}\kern-.6\wd0}}

\DeclareMathOperator{\diam}{diam}
\DeclareMathOperator{\diff}{d}

\DeclareMathOperator\Mat{Mat}

\newcommand\SO{{\rm{SO}}}
\newcommand\Lie{{\rm Lie}}

\def\sl{{\mathfrak{sl}}}

\def\bbr{\mathbb{R}}

\def\bbn{\mathbb{N}}

\def\R{\bbr}

\def\N{\bbn}
\def\rfrak{\mathfrak{r}}

\def\tbf{\mathbf{t}}

\def\wbf{\mathbf{w}}

\def\vbf{\mathbf{v}}

\def\vare{\varepsilon}

\def\zg0{Z_{G_\omega}(s)}

\def\zg{Z_G(s)}

\def\be{\begin{equation}}
\def\ee{\end{equation}}

\newcommand{\rhsc}{\delta}

\newcommand {\absolute}[1] {\left| {#1} \right|}
\newcommand {\norm}[1] {\left\| {#1} \right\|}

\newcommand\eng{\mathcal E}
\newcommand\egbd{C}

\newcommand{\hide}[1]{}

\newcommand{\sqf}{Q_0}

\title{Projection Theorems and Isometries of Hyperbolic Spaces}
\author{K.\ W.\ Ohm}
\date{}

\address{Department of Mathematics, University of California, San Diego, CA 92093}
\email{kwohm@ucsd.edu}
\thanks{}

\begin{document}

\maketitle

\begin{abstract}
    We prove a restricted projection theorem for an $n-2$ dimensional family of projections from $\R^n$ to $\R$. 
    
    The family we consider arises naturally in the context of the adjoint representation of the maximal unipotent subgroup of $\SO(n-1,1)$ on the Lie algebra of $\SO(n,1)$.
\end{abstract}

\section{Introduction}\label{sec: introduction}
The classical Marstrand projection theorem states that for a compact subset $K\subset\R^n$ and for a.e.\ $\vbf\in\mathbb S^{n-1}$
\be\label{eq: Marstrand}
\dim{\rm p}_{\vbf}(K)=\min(1,\dim K),  
\ee
where ${\rm p}_\vbf(X)=X\cdot \vbf$ is the orthogonal projection in the direction of $\vbf$ and here and in what follows $\dim$ denotes the Hausdorff dimension. Analogous statements hold more generally for orthogonal projection into a.e.\ $m$-dimensional subspace, with respect to the Lebesgue measure on ${\rm Gr}(m, n)$.  

The question of obtaining similar results as in~\eqref{eq: Marstrand} where $\vbf$ is confined to a proper Borel subset $B\subset \mathbb S^{n-1}$ has also been much studied, e.g., by Mattila, Falconer, Bourgain and others. Note, however, that without further restrictions on $B$,~\eqref{eq: Marstrand} fails: e.g., if 
\[
B=\{(\cos t, \sin t, 0): 0\leq t\leq 2\pi\}
\] 
is the great circle in $\mathbb S^2$ and $K$ is the $z$-axis, then ${\rm p}_\vbf(K)=0$ for every $\vbf\in B$. 

It was conjectured by F\"assler and Orponen~\cite{FaOr} that these are essentially the only type of obstructions; more precisely, they conjecture that if $\gamma:[0,1]\to\mathbb S^2$ is a curve so that for all $\{\gamma(t),\gamma'(t),\gamma''(t)\}$ span $\R^3$ for all $t$, then for a.e.\ $t\in[0,1]$, 
\[
\dim{\rm p}_{\gamma(t)}(K)=\min(1,\dim K).
\]
This conjecture was recently proved by~\cite{PYZ}; see also the earlier work \cite{kenmki2017marstrandtype} which relies on similar techniques as~\cite{PYZ}, and the more recent work \cite{GGW} which uses different techniques --- a major difficulty here is the failure of transversality in sense of~\cite{SchPr}. 

In this paper we consider a restricted projection problem in the same vein, which is motivated by recent applications in homogeneous dynamics, see \S\ref{sec: SO(n,1)} for more details. 

Let us fix some notation in order to state the main results. 
Let $n\geq 3$. We use the following coordinates for $\R^n$
\[
\R^{n}=\{(r_1,\wbf, r_2): r_i\in \R, \wbf\in\R^{n-2}\}. 
\]
Let $L:\R^{n-2}\to\R^{n-2}$ be an isomorphism, and let $q:\R^{n-2}\to\R$ be a positive definite quadratic form. For every $\tbf\in\R^{n-2}$, define $\pi_\tbf=\pi_{L, q, \tbf}:\R^{n}\to\R$ by 
\[
\pi_{\tbf}(r_1, \wbf, r_2)= r_1 + \wbf\cdot L(\tbf) + r_2q(\tbf)
\]
where $\wbf\cdot L(\tbf)$ is the usual inner product on $\R^{n-2}$.

\medskip

In this paper, we prove the following theorem.

\begin{thm}\label{thm: main}
    Let $K\subset \R^n$ be a compact subset. Then for almost every $\tbf\in\R^{n-2}$, we have 
    \[
    \dim(\pi_\tbf(K))=\min(1,\dim K)
    \]
\end{thm}

Indeed for most applications a discretized version of Theorem~\ref{thm: main} is required. 
This is the content of the following theorem.

\begin{thm}\label{thm: main finitary}
    Let $0<\alpha\leq 1$, and let  $0<\delta_0\leq1$. 
Let $F\subset B_{\R^n}(0,1)$ be a finite set satisfying the following: 
\[
\#(B_{\R^n}(X, \delta)\cap F)\leq C\cdot \delta^\alpha\cdot (\# F)\quad\text{for all $X\in F$ and all $\delta\geq \delta_0$}
\]
where $C\geq 1$.

Let $0<\vare<\alpha/100$.
For every $\delta\geq \rhsc_0$, there exists a subset 
$B_{\delta}\subset B:=\{\tbf\in\R^{n-2}: 1\leq \norm{\tbf}\leq 2\}$ 
with 
\[
|B\setminus B_{\delta}|\ll \vare^{-A}\delta^{\vare}
\]
so that the following holds. 
Let ${\bf t}\in B_{\delta}$, then there exists $F_{\delta,{\bf t}}\subset F$ with 
\[
\#(F\setminus F_{\delta,{\bf t}})\ll  \vare^{-A}\delta^{\vare}\cdot (\# F)
\]
such that for all $X\in F_{\delta,{\bf t}}$, we have 
\[
\#\Bigl(\{X'\in F_{\delta, \tbf}: |\pi_{\bf t}(X')-\pi_{\bf t}(X)|\leq \delta\}\Bigr)\ll C \delta_0^{-10}\cdot \delta^{\alpha}\cdot (\# F),
\] 
where $A$ is absolute and the implied constants depend on $L$ and $q$. 
\end{thm}

\begin{remark}
    Throughout the paper, the notation $a\ll b$ means $a\leq Db$ for a positive constant $D$ whose dependence varies and is explicated in different statements. 
\end{remark}

The most difficult case of the above theorem is arguably the case $n=3$, which was studied in~\cite{kenmki2017marstrandtype, PYZ} using fundamental works of Wolff and Schlag~\cite{Wolff, Schlag} --- see also~\cite{GGW} for a different approach to a more general problem in the same vein. 

Indeed when $n>3$, it is possible to deduce Theorem~\ref{thm: main finitary} using techniques of~\cite{SchPr} more directly. We, however, take a slightly different route which utilizes a hybrid of these two methods.  
       
\subsection*{Acknowledgment}
We would like to thank Amir Mohammadi for suggesting the problem and for helpful conversations. 

%%%%%%%%%%%%%%%%%
%%%%%%%%%%%%%%%%%

\section{Proof of Theorem~\ref{thm: main}}\label{sec: proof of main}

Theorem~\ref{thm: main} can be proved using the finitary version Theorem~\ref{thm: main finitary}. However, 
for the convenience of the reader, we present a self contained proof of Theorem~\ref{thm: main} in this brief section. This will also help explain the main idea of the proof of Theorem~\ref{thm: main finitary}. 

\medskip

Let $L$ and $q$ be as in \S\ref{sec: introduction}. 
For every $\tbf\in\R^{n-2}$, define 
\be\label{eq: def ft}
f_\tbf=f_{L,q,\tbf}:\R^n\to\R^3\quad\text{by}\quad f_\tbf(r_1,\wbf, r_2)= \Bigl(r_1, \wbf \cdot L(\tbf), r_2q(\tbf) \Bigr);
\ee
recall our coordinates $\R^{n}=\{r_1,\wbf, r_2): r_i\in\R, \wbf\in\R^{n-2}\}$.

\begin{lemma}\label{lem: transversality of ft}
Let $0<\vare<0.01$. 
Let $X, X'\in\R^n$ satisfy that $\norm{X-X'}=1$. Then   
\[
\absolute{\{\tbf\in B: \norm{f_\tbf(X)-f_\tbf(X')}\leq \vare\}}\ll \vare
\]
where $B=\{\tbf\in\R^{n-2}: 1\leq \norm{\tbf}\leq 2\}$ and the implied constant depend on $L$ and $q$.    
\end{lemma}

\begin{proof}
    Let us write $X=(r_1,\wbf, r_2)$ and $X'=(r'_1,\wbf', r_2')$. Let us denote the set in the lemma by $S$. 
    If $S\neq \emptyset$, then there exists some $\tbf\in B$ so that 
    \[
    \norm{\Bigl(r_1-r'_1, (\wbf-\wbf')\cdot L(\tbf), (r_2-r'_2)q(\tbf)\Bigr)}\leq \vare.
    \]
    Thus $\absolute{r_1-r_1'}, \absolute{r_2-r_2'}q(\tbf)\leq \vare\leq 1/10$. This and $\norm{X-X'}=1$ imply 
    \[
    \norm{\wbf-\wbf'}\gg 1
    \]
    where the implied constant depends on $\min_{\norm{\tbf}=1} \absolute{q(\tbf)}$. 
    
    Altogether, either $S=\emptyset$ in which case the proof is complete, or we may assume $\norm{\wbf-\wbf'}\gg1$ and
    \[
    S\subset \{\tbf\in B: \norm{(\wbf-\wbf')\cdot L(\tbf)}\leq \vare\}.
    \]
    Since $\norm{\wbf-\wbf'}\gg 1$, the measure of the set on the right side of the above is $\ll\vare$, where the implied constant depends on $L$ and $q$. 
    
    The proof is complete. 
\end{proof}

\begin{lemma}\label{lem: ft dim preserving}
 Let $0<\alpha<1$. Let $\mu$ be a probability measure supported on $B_{\R^n}(0,1)$ which satisfies 
 \[
 \eng_\alpha(\mu)\leq C
 \]
 for some $C\geq 1$. 
 Then for every $R>0$, we have 
 \[
 \absolute{\{\tbf\in\R^{n-2}: 1\leq \norm{\tbf}\leq 2, \eng_\alpha(f_\tbf\mu)>R\}}\leq C'/R
 \]
 where $C'\ll \frac{C}{1-\alpha}$. 
\end{lemma}

\begin{proof}
We recall the standard argument which is based on Lemma~\ref{lem: transversality of ft}.

Put $\mu_\tbf=f_\tbf\mu$. Using the definition of 
$\alpha$-dimensional energy and the Fubini's theorem, we have 
\begin{align*}
    \int_B\eng_\alpha(\mu_\tbf)\diff\!\tbf&= \int_B\int_{\R^n}\int_{\R^n}\frac{\diff\!\mu_\tbf(X)\diff\!\mu_\tbf(X')}{\norm{f_\tbf(X)-f_\tbf(X')}^\alpha}\diff\!\tbf\\
    &=\int_{\R^n}\int_{\R^n}\int_B\frac{\diff\!\tbf}{\norm{f_\tbf(X)-f_\tbf(X')}^\alpha} \diff\!\mu_\tbf(X)\diff\!\mu_\tbf(X')
\end{align*}
Renormalizing with the factor $\norm{X-X'}^{\alpha}$, we conclude that
\be\label{eq: int energy mut}
\int_B\eng_\alpha(\mu_\tbf)\diff\!\tbf=\int_{\R^n}\int_{\R^n}\int_B\frac{\diff\!\tbf}{\norm{\frac{f_\tbf(X)-f_\tbf(X')}{\norm{X-X'}}}^\alpha}
\frac{\diff\!\mu_\tbf(X)\diff\!\mu_\tbf(X')}{\norm{X-X'}^\alpha}.
\ee
Since $\alpha<1$, applying Lemma~\ref{lem: transversality of ft}, we conclude that
\[
\int_B  \norm{\frac{f_\tbf(X)-f_\tbf(X')}{\norm{X-X'}}}^{-\alpha}\diff\!\tbf\ll \frac{1}{1-\alpha}.
\]
This,~\eqref{eq: int energy mut}, and $\eng_\alpha(\mu)\leq C$ imply that   
\[
\int_B\eng_\alpha(\mu_\tbf)\diff\!\tbf\ll\frac{\eng_\alpha(\mu)}{1-\alpha}\ll\frac{C}{1-\alpha}.
\]
The claim in the lemma follows from this and the Chebyshev's inequality. 
\end{proof}

\begin{proof}[Proof of Theorem~\ref{thm: main}]
Let $s\in\R$ and $\tbf\in\R^{n-2}$. Then 
\be\label{eq: decomp pi t}
\begin{aligned}
    \pi_{s\tbf}(r_1,\wbf, r_2)&=r_1+\wbf\cdot L(s\tbf)+r_2q(s\tbf)\\
    &=(1, s, s^2).f_\tbf(r_1,\wbf, r_2). 
\end{aligned}
\ee

Let $K\subset B_{\R^n}(0,1)$ be a compact subset, and let  $\kappa=\min(1,\dim K)$. Let $0<\alpha<\kappa$. 
By Frostman's lemma, there exists a probability measure $\mu$ supported on $K$ and satisfying the following   
\[
\mu(B(X,\delta))\leq \delta^\alpha \quad\text{for all $X\in K$}.
\]
Then by Lemma~\ref{lem: ft dim preserving}, applied with $\mu$, there exists a conull subset $\Xi_\alpha\subset\R^{n-2}$ so that  
$\dim(f_\tbf(K))\geq \alpha$ for all $\tbf \in\Xi_\alpha$. Applying this with $\alpha_n=\kappa-\frac1n$ for all $n\in\N$, we obtain a conull subset $\Xi\subset\R^{n-2}$ so that 
\[
\dim(f_\tbf(K))\geq \kappa,\quad\text{for all $\tbf \in\Xi$.}
\]

Let $\tbf\in\Xi$, and set $K_\tbf=f_\tbf(K)$. Then by~\cite[Thm.\ 1.3]{PYZ}, see also~\cite{GGW}, for a.e.\ $s\in\R$, we have 
\[
\Bigl\{x_1+x_2s+x_3s^2: (x_1, x_2, x_3)\in K_\tbf\Bigr\}\subset \R
\]
has dimension $\kappa$. This and~\eqref{eq: decomp pi t} complete the proof.  
\end{proof}

%%%%%%%%%%%%%
%%%%%%%%%%%%%

\section{Proof of Theorem~\ref{thm: main finitary}}\label{sec: proof of finitary}

We now turn to the proof of Theorem~\ref{thm: main finitary}, the argument is a discretized version of the argument in \S\ref{sec: proof of main} as we now explicate. 

Let $F\subset\R^n$ be a finite, and let  $\mu$ be the uniform measure on $F$. Our standing assumption is that for some $0< \alpha\leq 1$ and some $\egbd\geq 1$, we have  
\be\label{eq: dimension alpha}
\mu(B_{\R^n}(X, \delta))\leq C\rhsc^\alpha\qquad\text{for all $X\in F$ and all $\delta\geq \delta_0$}.
\ee

Without loss of generality, we will assume $\delta_0=2^{-k_0}$ for some $k_0\in\N$; we will also assume that $\diam(F)\leq 1$. 

For a finitely supported probability measure $\rho$ on $\R^d$, define  
\[
\eng_{\alpha,\rho}^+:\R^d\to\R\quad\text{by}\quad \hat\eng_{\alpha,\rho}(X)=\int \norm{X-X'}_+^{-\alpha}\diff\!\rho(X')
\]
where $\norm{X-X'}_+=\max\Bigl(\norm{X-X'},\delta_0\Bigr)$ for all $X,X'\in\R^d$ --- this definition is motivated by the fact that we are only concerned with scales $\geq \delta_0$, 

\medskip

We recall the following standard lemma. 

\begin{lemma}\label{lem: truncated energy}
    Let $\rho$ be a finitely supported probability measure on $\R^d$. Assume that for some $X\in \R^d$ we have 
    \[
    \hat\eng_{\alpha,\rho}(X)\leq R.
    \]
    Then for all $\delta\geq\delta_0$, we have 
    \[
    \rho(B_{\R^d}(X,\delta))\leq R\delta^{\alpha}.  
    \]
\end{lemma}

\begin{proof}
We include the proof for completeness. 
Let $\delta\geq \delta_0$, then  
\begin{align*}
    \delta^{-\alpha}\rho(B_{\R^d}(X,\delta))&\leq \int_{B_{\R^d}(X,\delta)}\norm{X-X'}_+^{-\alpha}\diff\!\rho(X')\\
    &\leq \int \norm{X-X'}_+^{-\alpha}\diff\!\rho(X')=\hat\eng_{\alpha,\rho}(X)\leq R,
\end{align*}
as it was claimed. 
\end{proof}

Recall our notation $B=\{\tbf\in\R^{n-2}: 1\leq \tbf\leq 2\}$. For every $\tbf\in B$, let $\mu_\tbf=f_\tbf\mu$ 
where $f_\tbf:\R^n\to\R$ is defined as in~\eqref{eq: def ft}:
\be\label{eq: def ft'}
f(r_1,\wbf, r_2)= \Bigl(r_1, \wbf \cdot L(\tbf), r_2q(\tbf) \Bigr),
\ee
and $\R^{n}=\{r_1,\wbf, r_2): r_i\in\R, \wbf\in\R^{n-2}\}$.

\begin{lemma}\label{lem: ave trun egy proj}
    For every $X\in F$,  
    \[
    \int_B\hat\eng_{\alpha,\mu_\tbf}(f_\tbf X)\diff\!\tbf\ll C\absolute{\log_2(\delta_0)}
    \]
    where the implied constant is absolute. 
\end{lemma}

\begin{proof}
    Let $X\in F$. By the definitions, we have 
    \[
    \int_B\hat\eng_{\alpha,\mu_\tbf}(f_\tbf X)\diff\!\tbf= \int_B\int \norm{f_\tbf X-f_\tbf X'}_+^{-\alpha}\diff\!\mu(X')\diff\!\tbf.
    \]
   Renormalizing with $\norm{X-X'}_+^{-\alpha}$ and using Fubini's theorem, we have 
   \[
   \int_B\hat\eng_{\alpha,\mu_\tbf}(f_\tbf X)\diff\!\tbf=\int \int_B\frac{1}{\frac{\norm{f_\tbf X-f_\tbf X'}_+^\alpha}{\norm{X-X'}_+^\alpha}}\diff\!\tbf \norm{X-X'}_+^{-\alpha}\diff\!\mu(X').
   \]
  For every $0\leq k\leq k_0-1$, let 
 \[
 F_k(X)=\{X'\in F: 2^{-k-1}< \norm{X-X'}\leq 2^{-k}\},
 \]
 and let $F_{k_0}(X)=\{X'\in F: \norm{X-X'}\leq 2^{-k_0}\}$. 
 
 For all $X'\in F_{k_0}$, we have $1\leq \frac{\norm{f_\tbf X-f_\tbf X'}_+}{\norm{X-X'}_+}\leq 10$. Thus, 
   \be\label{eq: contribution of Fk0}
   \int_{F_{k_0}}\int_B\frac{1}{\frac{\norm{f_\tbf X-f_\tbf X'}_+^\alpha}{\norm{X-X'}_+^\alpha}}\diff\!\tbf \norm{X-X'}_+^{-\alpha}\diff\!\mu(X')\ll \mu(F_{k_0})2^{k_0\alpha}\ll C. 
   \ee
   We now turn to the contribution of $F_k$ to the above integral for $k< k_0$. 
   If $X'\in F_k$, for some $k<k_0$, then $\norm{X-X'}_+=\norm{X-X'}$ and we have 
   \begin{multline}\label{eq: contribution of Fk}
   \int_{F_{k}}\int_B\frac{1}{\frac{\norm{f_\tbf X-f_\tbf X'}_+^\alpha}{\norm{X-X'}_+^\alpha}}\diff\!\tbf \norm{X-X'}_+^{-\alpha}\diff\!\mu(X')= \\
   \int_{F_{k}}\int_B\frac{1}{\frac{\norm{f_\tbf X-f_\tbf X'}_+^\alpha}{\norm{X-X'}^\alpha}}\diff\!\tbf \norm{X-X'}^{-\alpha}\diff\!\mu(X'). 
   \end{multline}
   By Lemma~\ref{lem: transversality of ft}, we have 
   \[
   \int_B\frac{1}{\frac{\norm{f_\tbf X-f_\tbf X'}_+^\alpha}{\norm{X-X'}^\alpha}}\diff\!\tbf \leq \int_B\frac{1}{\norm{\frac{f_\tbf X-f_\tbf X'}{\norm{X-X'}}}^\alpha}\diff\!\tbf\ll 1
   \]
   where the implied constant is absolute. Thus 
   \[
   \eqref{eq: contribution of Fk}\ll \mu(F_k)2^{k\alpha}\ll C
   \]
   This and~\eqref{eq: contribution of Fk0} imply that 
   \[
   \int_B\hat\eng_{\alpha,\mu_\tbf}(f_\tbf(X))\diff\!\tbf\ll Ck_0=C\absolute{\log_2(\delta_0)}
   \]
   as we claimed. 
\end{proof}

\begin{propos}\label{prop: ft preserves dim}
Let $0<\vare<1$. 
    There exists $B'\subset B$ with 
    \[
    \absolute{B\setminus B'}\leq \vare^{-A'}\delta_0^{\vare}
    \]
    so that the following holds. For every $\tbf\in B'$, there exists $F_{\tbf}\subset F$ with 
    \[
    \mu(F\setminus F_\tbf)\leq \vare^{-A'}\delta_0^{\vare}
    \]
    so that for every $X\in F_\tbf$ and every $\delta\geq \delta_0$ we have 
    \[
    \mu_\tbf\Bigl(B_{\R^3}(f_\tbf(X), \delta)\Bigr)\ll \vare^{-A'}\delta_0^{-3\vare}\cdot \delta^{\alpha}
    \]
    where $A'$ is absolute and the implied constant depends on $C$. 
\end{propos}

\begin{proof}
  The proof is based on Lemma~\ref{lem: ave trun egy proj} and Chebychev's inequality as we now explicate. First note that we may assume $\delta_0$ is small enough (polynomially in $\vare$) so that 
  \[
  \delta_0^{-\vare/10}>\absolute{\log\delta_0}
  \]
  otherwise the statement follows trivially. 
  
  By Lemma~\ref{lem: ave trun egy proj}, we have 
  \be\label{eq: lem ave use}
\int_B\hat\eng_{\alpha,\mu_\tbf}(f_\tbf(X))\diff\!\tbf\leq C'\absolute{\log_2(\delta_0)}.
  \ee
  where $C'\ll C$. Averaging~\eqref{eq: lem ave use}, with respect to $\mu$, and using Fubini's theorem we have 
  \be\label{eq: ft pres dim 1}
  \int_B\int\hat\eng_{\alpha,\mu_\tbf}(f_\tbf(X))\diff\!\mu(X)\diff\!\tbf\leq C'\absolute{\log_2(\delta_0)}.
  \ee
  Let $B'=\{\tbf\in B: \int\hat\eng_{\alpha,\mu_\tbf}(f_\tbf(X))\diff\!\mu(X)<C'\delta_0^{-2\vare}\}$.
  Then by~\eqref{eq: ft pres dim 1} and Chebychev's inequality we have 
  \[
  \mu(B\setminus B')\leq \delta_0^{\vare}.
  \] 
  Let now $\tbf\in B'$, then 
  \be\label{eq: int of energy for good t}
  \int\hat\eng_{\alpha,\mu_\tbf}(f_\tbf(X))\diff\!\mu(X)\leq C'\delta_0^{-2\vare}.
  \ee
  For every $\tbf\in B'$, set 
  \[
  F_\tbf=\{X\in F: \hat\eng_{\alpha,\mu_\tbf}(f_\tbf(X))< C'\delta_0^{-3\vare}\}
  \]
  Then~\eqref{eq: int of energy for good t} and Chebychev's inequality again imply that $\mu(F\setminus F_\tbf)\leq \delta_0^{\vare}$.

  Altogether, for every $\tbf\in B'$ and $X\in F_\tbf$, we have 
  \[
  \hat\eng_{\alpha,\mu_\tbf}(f_{\tbf}(X))=\int\norm{f_\tbf(X)-f_\tbf(X')}_+^{-\alpha}\diff\!\mu_\tbf(X')\leq C'\delta_0^{-3\vare}
  \]
This and Lemma~\ref{lem: truncated energy} imply that for every $\tbf\in B'$ and $X\in F_\tbf$, we have
\[
\mu_\tbf\Bigl(B_{\R^3}(f_\tbf(X),\delta)\Bigr)\leq C'\delta_0^{-3\vare}\cdot \delta^\alpha \quad\text{for all $\delta\geq \delta_0$}.
\]
The proof is complete. 
\end{proof}

\subsection*{Proof of Theorem~\ref{thm: main finitary}}
We now turn to the proof of Theorem~\ref{thm: main finitary}. As it was done in the proof of Theorem~\ref{thm: main}, we will use the following observation: for all $s\in\R$ and $\tbf\in\R^{n-2}$, we have  
\be\label{eq: decomp pi t'}
\begin{aligned}
    \pi_{s\tbf}(r_1,\wbf, r_2)&=r_1+\wbf\cdot L(s\tbf)+r_2q(s\tbf)\\
    &=(1, s, s^2).f_\tbf(r_1,\wbf, r_2). 
\end{aligned}
\ee 

Apply Proposition~\ref{prop: ft preserves dim} with $\vare$ as in the statement of Theorem~\ref{thm: main finitary}. 
Let $B'\subset B$ be as in that proposition, and for every $\tbf\in B'$, let $F_{\tbf}$ be as in that proposition. 
Then we have 
\be\label{eq: lem ave trun egy proj use}
\mu_\tbf\Bigl(B_{\R^3}(f_\tbf(X),\delta)\Bigr)\leq C' \vare^{-A'}\delta_0^{-3\vare}\cdot\delta^\alpha\quad\text{for all $X\in F_\tbf$ and $\delta\geq\delta_0$}.
\ee

Let $K_\tbf= f_\tbf(F_\tbf)\subset \R^3$ and let $\rho_\tbf$ be the restriction of $\mu_t$ to $K_\tbf$ normalized to be a probability measure. Then~\eqref{eq: lem ave trun egy proj use} and the fact that $\mu(F\setminus F_{\tbf})\leq \delta_0^{\vare}$ imply that   
\be\label{eq: lem ave trun egy proj use'}
\rho_\tbf\Bigl(B_{\R^3}(Y,\delta)\Bigr)\leq 2C' \vare^{-A'}\delta_0^{-3\vare}\cdot\delta^\alpha\quad\text{for all $Y\in K_\tbf$ and $\delta\geq\delta_0$}.
\ee
This in particular implies that $K_\tbf$ and  $\rho_\tbf$ satisfy the conditions in~\cite[Thm.~B]{LM-PolyDensity}, see also \cite{PYZ} and \cite[Thm.\ 2.1]{GGW}. Apply~\cite[Thm.~B]{LM-PolyDensity} with $\vare$; thus, there there exists $J_{\delta,\tbf}\subset [0,2]$ with 
\[
\absolute{[0,2]\setminus J_{\delta,\tbf}}\leq \hat C \vare^{-D}\delta^{\vare}
\]
and for all $s\in J_{\delta,\tbf}$ there is a subset $K_{\delta, \tbf, s}\subset K_{\tbf}$ with 
\[
\rho_\tbf(K_{\tbf}\setminus K_{\delta, \tbf, s}) \leq \hat C\vare^{-D}\delta^{\vare}
\]
so that for all $Y\in K_{\delta, \tbf, s}$, we have 
\[
\rho_\tbf\Bigl(\{Y'\in K_\tbf: \absolute{(1,s,s^2)\cdot (Y-Y')}\leq \delta\}\Bigr)\leq \hat C \vare^{-D}\delta_0^{-3\vare} \cdot\delta^{\alpha- 7\vare};
\]

Let $A=\max\{A', D\}$.
In view of the definition of $\rho_\tbf$ and~\eqref{eq: decomp pi t'}, we have the following. For every $\tbf\in B'$ 
and $s\in J_{\delta,\tbf}$, put $F_{\delta, s\tbf}=F\cap f^{-1}_\tbf(K_{\delta, \tbf, s})$. Then 
\[
\#(F\setminus F_{\delta, s\tbf})\leq  10\hat C\vare^{-A}\delta^{\vare}\cdot (\#F),
\]
and for every $X\in F_{\delta, s\tbf}$, we have 
\be\label{eq: F delta st}
\#\{X'\in F_{\delta, s\tbf}: \absolute{\pi_{s\tbf}(X)-\pi_{s\tbf}(X')}\leq \delta\}\Bigr)\leq \hat C \vare^{-A}\delta_0^{-3\vare} 
\delta^{\alpha- 7\vare}. 
\ee

This finishes the proof. Indeed, let $B_\delta\subset B$ be the set of $\tbf\in  B$ 
for which there exists $F_{\delta, \tbf}\subset F$ with
\[
\#(F\setminus F_{\delta, \tbf})\leq 100\hat C\vare^{-D}\delta^{\vare}\cdot (\#F)
\] 
so that for all $X\in F_{\delta, \tbf}$ we have
\[
\#\{X'\in F_{\delta, \tbf}: \absolute{\pi_{\tbf}(X)-\pi_{\tbf}(X')}\leq \delta\}\Bigr)\leq \hat C \vare^{-A}\delta_0^{-10\vare} 
\delta^{\alpha}. 
\]
Then~\eqref{eq: F delta st} implies that for every $\tbf'\in B'$ and $s\in J_{\delta,\tbf'}$, we have $s\tbf'\in B_\delta$, so long as $s\tbf'\in B$. In particular, we conclude that 
\[
\absolute{B\setminus B_\delta}\ll \vare^{-A}\delta^{\vare}
\] 
as we aimed to prove.
\qed

%%%%%%%%%%%%%%%%
%%%%%%%%%%%%%%%%

\section{Theorem~\ref{thm: main finitary} and the isometry group of $\mathbb H^n$}\label{sec: SO(n,1)}

Let $n\geq 3$, and let 
\[
\sqf(x_1,x_1,\ldots, x_{n+1})=2x_1x_{n+1}-\sum_{i=1}^{n} x_i^2
\]
Then $G=\SO(\sqf)^\circ\simeq\SO(n,1)^\circ$ is the group of orientation preserving isometries of $\mathbb H^{n}$. Let 
\[
\Lie(G)=\{A\in\sl_{n+1}(\R): A^T\sqf+\sqf A=0\}
\]
where we identify $\sqf$ and the corresponding symmetric matrix, i.e., 
\[
\sqf= \begin{pmatrix} 0 & 0 & 1\\ 
0 & -I_{n-2} & 0\\ 
1 & 0 & 0 \end{pmatrix}.
\]

Let $H\subset G$ be the stabilizer of $e_{n}=(0,\ldots, 1, 0)$. Then 
\[
\Lie(G)=\Lie(H)\oplus \rfrak
\]
where $\rfrak$ is invariant under conjugation by $H$ and $\dim\rfrak=n$. 
More explicitly, 
\[
\rfrak=\Bigl\{X(r_1, \wbf, r_2): r_1,r_2\in \R, \wbf\in\R^{n-2}\Bigr\}
\]
where for $r_1, r_2\in\R$ and $\wbf \in \R^{n-2}$,
\be\label{eq: def X}
X(r_1, \wbf, r_2)=\begin{pmatrix} 0 & \begin{array}{cc} {\bf 0} & r_1 \end{array} & 0\\ 
\begin{array}{c}
    {\bf 0}^T  \\
     -r_2 
\end{array} & 
R(\wbf) & 
\begin{array}{c} {\bf 0}^T 
\\ -r_1 
\end{array}\\ 
0 & \begin{array}{c c}
    {\bf 0} & r_2 
\end{array} 
& 0 \end{pmatrix}\in\Mat_{n+1}(\R)
\ee
here ${\bf 0}\in\R^{n-2}$ and $R(\wbf)=\begin{pmatrix} 
{\bf 0}_{n-2} & \wbf^T \\ 
-\wbf &0 
\end{pmatrix}\in\Mat_{n-1}(\R)$. 

We may identify $\rfrak$ with $\R^{n}$ using the above coordinates. 
With this notation, put 
\[
\rfrak^+=\{X(r_1,0,0): r_1\in\R\}\simeq \R.
\]

Define the subgroup $U\subset H$ as follows. 
For every $\tbf\in\R^{n-2}$, let 
\[
u_{\tbf}=\begin{pmatrix} 1 & \begin{array}{cc} {\bf t} & 0 \end{array} & \tfrac12\norm{\bf t}^2\\ 
{\bf 0}^T & I_{n-2} & 
\begin{array}{c} {\bf t}^T 
\\ 0 
\end{array}\\ 
0 & {\bf 0} & 1 \end{pmatrix}\in\Mat_{n+1}(\R).
\]
where ${\bf 0}\in\R^{n-1}$. Put 
\[
U=\{u_\tbf: \tbf\in\R^{n-2}\}. 
\]

It is worth noting that if $a_s$ denotes the one parameter diagonal subgroup of $H$ defined by
\[
a_se_1=e^se_1,\quad a_s e_{n+1}=e^{-s}e_{n+1}, \quad a_se_i=e_i\quad 2\leq i\leq n,
\]
then 
\[
\begin{aligned}
U&=\{h\in H: \lim_{s\to-\infty}a_sha_{-s}=1\}\quad\text{and}\\
\rfrak^+&=\{X\in\rfrak: \lim_{s\to-\infty}a_sXa_{-s}=0\}.
\end{aligned}
\]

For every $\tbf\in\R^{n-2}$, define 
\[
\xi_{\tbf}:\rfrak\to\rfrak^+\quad\text{by}\quad\xi_{\tbf}(X)=(u_{\tbf} Xu_{-\tbf})^+. 
\]
where for $X\in\rfrak$, $X^+$ denote the orthogonal projection to $\rfrak^+$. 

\medskip

We have the following lemma. 

\begin{lemma}\label{lem: matrix multiplication}
Identify $\rfrak^+$ and $\R$ as above. Then   
   \[
\xi_{\tbf}(X(r_1,\wbf, r_2)) = \pi_{{\rm id}, q_{\rm st},\tbf}(r_1,\wbf, r_2)
\] 
where ${\rm id}:\R^{n-2}\to\R^{n-2}$ is the identity map and $q_{\rm st}(\tbf)=\frac12\norm{\tbf}^2$. 
\end{lemma}

\begin{proof}
    The proof is based on a direct computation as we now explicate. 
    
    To simplify the notation slightly, we will write $\tilde\tbf=(\tbf, 0)$ and $\tilde r_i=({\bf 0}, r_i)$. Note that $\tilde\tbf\cdot\tilde r_i=0$. 
    Recall also that 
    \[
   \xi_\tbf(X)=(u_\tbf X(r_1, r_2, \wbf) u_{-\tbf})^+
    \]
    We have  
    \begin{multline*}
   u_\tbf X(r_1, r_2, \wbf) =\begin{pmatrix} 1 &  
    \tilde\tbf  & 
    \tfrac12\norm{\bf t}^2\\ 
    {\bf 0}^T & I_{n-2} & \tilde{\bf t}^T \\ 
0 & {\bf 0} & 1 \end{pmatrix} 
\begin{pmatrix} 0 & \tilde r_1 & 0\\ 
-\tilde{r}_2^T  & 
R(\wbf) & 
-\tilde{r}_1\\ 
0 & \tilde r_2 
& 0 \end{pmatrix}=\\ 
\begin{pmatrix} 0 & \tilde r_1+\tilde\tbf R(\wbf)+\tfrac12\norm{\tbf}^2\tilde r_2 & 0\\ 
-\tilde{r}_2^T  & 
R(\wbf)+\tilde\tbf^T\tilde r_2 & 
-\tilde{r}_1\\ 
0 & \tilde r_2 
& 0 \end{pmatrix}
    \end{multline*}
    Multiplying this with $u_{-\tbf}$, we have 
    \begin{multline*}
      u_\tbf X(r_1, r_2, \wbf)u_{-\tbf}= \\
      \begin{pmatrix} 0 & \tilde r_1+\tilde\tbf R(\wbf)+\tfrac12\norm{\tbf}^2\tilde r_2 & 0\\ 
-\tilde{r}_2^T  & 
R(\wbf)+\tilde\tbf^T\tilde r_2 & 
-\tilde{r}_1\\ 
0 & \tilde r_2 
& 0 \end{pmatrix} 
\begin{pmatrix} 1 &  
    -\tilde\tbf  & 
    \tfrac12\norm{\bf t}^2\\ 
    {\bf 0}^T & I_{n-2} & -\tilde{\bf t}^T \\ 
0 & {\bf 0} & 1 \end{pmatrix} =\\
\begin{pmatrix} 0 & \tilde r_1+\tilde\tbf R(\wbf)+\tfrac12\norm{\tbf}^2\tilde r_2 & *\\ 
*  & * & *\\ 
* & * & * \end{pmatrix}
    \end{multline*}
   Since $\tilde\tbf R(\wbf)=({\bf 0}, \tbf\cdot\wbf)$, the above and the definitions of $\xi_\tbf$ and $\pi_\tbf$ imply 
    \[
    \xi_\tbf(X)=r_1+\tbf\cdot\wbf+\tfrac12\norm{\tbf}^2 r_2=\pi_{{\rm id}, q_{\rm st},\tbf} (r_1,\wbf, r_2)
    \]
    as we claimed. 
\end{proof}

In view of Lemma~\ref{lem: matrix multiplication}, the following theorem is a restatement of Theorem~\ref{thm: main finitary} in the language of adjoint action of $U$ on $\rfrak$.

\begin{thm}\label{thm: SO(n,1)}
    Let $0<\alpha\leq 1$, and let  $0<\delta_0\leq1$. 
Let $F\subset B_{\rfrak}(0,1)$ be a finite set satisfying the following 
\[
\#(B_{\rfrak}(X, \delta)\cap F)\leq C\delta^\alpha\cdot (\# F)\qquad\text{for all $X\in F$ and all $\delta\geq \delta_0$}
\]
where $C\geq 1$.

Let $0<\vare<\alpha/100$.
For every $\delta\geq \delta_0$, there exists a subset $B_{\delta}\subset B:=\{\tbf\in\R^{n-2}: 1\leq \norm{\tbf}\leq 2\}$ with 
\[
|B\setminus B_{\delta}|\ll \vare^{-A}\delta^{\vare}
\]
so that the following holds. 
Let ${\bf t}\in B_{\delta}$, there exists $F_{\delta,{\bf t}}\subset F$ with 
\[
\#(F\setminus F_{\delta,{\bf t}})\ll \vare^{-A}\delta^{\vare}\cdot (\# F)
\]
such that for all $X\in F_{\delta,{\bf t}}$, we have 
\[
\#\Bigl(\{X'\in F_{\delta, \tbf}: |\xi_{\bf t}(X')-\xi_{\bf t}(X)|\leq \delta\}\Bigr)\leq C \delta_0^{-10}\cdot \delta^{\alpha}\cdot (\# F)
\] 
where $A$ and the implied constants are absolute. 
\end{thm}

\bibliographystyle{halpha}
\bibliography{papers}

\end{document}